\documentstyle[amsfonts, txfonts, amssymb, latexsym, 12pt]{amsart}

\newtheorem{theorem}{Theorem}[section]
\newtheorem{lemma}[theorem]{Lemma}

\def\Z{{\mbox{\rm\kern.25em
\vrule width.03em height0.57ex depth0ex
\kern.033em
\vrule width.03em height1.52ex depth-0.96ex \kern-.338em Z}}}
\def\R{{\mbox{\rm I\kern-.22em R}}}
\def\N{{\mbox{\rm I\kern-.22em N}}}

\def\supp{{\rm supp}}

\def\sgn{{\rm sgn}}
\def\M{{\cal{M}}}

\def\n{{\bf n}}

\def\M{{\cal{M}}}

\def\r{{\cal{R}}}

\def\111{\gamma}

\def\be#1{\begin{equation}\label{#1}}
\def\bas{\begin{align*}}
\def\eas{\end{align*}}
\def\bi{\begin{itemize}}
\def\ei{\end{itemize}}
\newenvironment{proof}{\noindent {\bf Proof} }{\endprf\par}
\def \endprf{\hfill  {\vrule height6pt width6pt depth0pt}\medskip}
\def\emph#1{{\it #1}}

\title[Calder\'{o}n commutators and Cauchy integral III]{Calder\'{o}n commutators and the Cauchy integral on Lipschitz curves revisited III. Polydisc extensions}

\author{Camil Muscalu}
\address{Department of Mathematics, Cornell University, Ithaca, NY 14853}
\email{camil@@math.cornell.edu}

\begin{document}

\begin{abstract}
This article is the last in a series of three papers, whose scope is to give new proofs to the well known theorems of Calder\'{o}n, Coifman, McIntosh and Meyer \cite{calderon}, \cite{meyerc}, \cite{cmm}. Here we extend the results of the previous two papers
to the polydisc setting. In particular, we solve completely an open question of Coifman from the early eighties.
\end{abstract}

\maketitle

\section{Introduction}

The present article is a natural continuation of the previous \cite{camil2}, \cite{camil3} and is the last paper in the sequel. The goal of it is to show that
the method developed in these papers to give new proofs to the $L^p$ boundedness of the Calder\'{o}n commutators and the Cauchy integral on Lipschitz curves \cite{calderon}, \cite{meyerc}, \cite{cmm}, can be used to extend these
classical results to the {\it $n$-parameter polydisc} setting, for any $n\geq 2$.

Suppose that $F$ is an analytic function on a disc of a certain radius centered at the origin in the complex plane and $A$ a complex valued function in $\R^n$, so that $\frac{\partial^n A}{\partial x_1 ... \partial x_n}\in L^{\infty}(\R^n)$ with
an $L^{\infty}$ norm strictly smaller than the radius of convergence of $F$. Define the linear operator $C_{n, F, A}$ by the formula

\begin{equation}\label{1}
C_{n, F, A} f(x) := p.v. \int_{\R^n} f(x+t ) F \left( \frac{\Delta_{t_1}^{(1)}}{t_1} \circ ... \circ\frac{\Delta_{t_n}^{(n)}}{t_n} A(x)\right)
\frac{d t_1}{t_1} ... \frac{d t_n}{t_n}
\end{equation}
for functions of $n$ variables $f(x)$ for which the principal value integral exists, where $\Delta_s^{(i)}$ denotes the finite difference operator at scale $s$ in the direction of $e_i$, given by

$$\Delta_s^{(i)} B(x) := B(x + s e_i) - B(x)$$
and $e_1$, ..., $e_n$ is the standard basis in $\R^n$.

The main theorem we are going to prove is the following.

\begin{theorem}\label{main}
The operator $C_{n, F, A}$ extends naturally as a bounded linear operator from $L^p(\R^n)$ into $L^p(\R^n)$ for every $1<p<\infty$.
\end{theorem}
This answers completely an open question of Coifman from the early eigthties \cite{j1}, \cite{j2}. The case when the $L^{\infty}$ norm of $\frac{\partial^n A}{\partial x_1 ... \partial x_n}$ is small and the generic $n=2$ case,
have been understood earlier by Journ\'{e} in \cite{j1} and \cite{j2} respectively. Our proof is quite different from the approach in \cite{j1}, \cite{j2} and works equally well in all dimensions. In fact, as we will describe in the last section of the paper,
much more can be proved in the same way. Not only the operators of (\ref{1}) are bounded, but also (for instance) those given by expressions of type

\begin{equation}\label{2}
f \rightarrow 
p.v. \int_{\R^4} f(x+t+s)
F\left(\frac{\Delta_{t_1}^{(1)}}{t_1}\circ\frac{\Delta_{t_2}^{(2)}}{t_2}\circ\frac{\Delta_{s_1}^{(1)}}{s_1}\circ\frac{\Delta_{s_2}^{(2)}}{s_2} A(x)\right)
\frac{d t_1}{t_1}\frac{d t_2}{t_2}\frac{d s_1}{s_1}\frac{d s_2}{s_2}
\end{equation}
and their natural generalizations. Of course, in (\ref{2}) one has to assume this time that $\frac{\partial^4 A}{\partial x_1^2 \partial x_2^2} \in L^{\infty}(\R^2)$. 
When $F(z) = z^d$ with $d\geq 1$ the operator in (\ref{1}) is the natural $n$-parameter extension of the $d$th Calder\'{o}n commutator, whereas 
for $F(z) = \frac{1}{1+i z}$ one obtains the $n$-parameter generalization of the Cauchy integral on Lipschitz curves \cite{calderon}, \cite{meyerc}, \cite{cmm}.

For simplicity, we shall denote from now on with $C_{n, d, A}$ the $n$-parameter $d$th Calder\'{o}n commutator. It is easy to observe that when $f(x)$ and $A(x)$ are particularly given by

$$f(x) = f_1(x_1) \cdot ... \cdot f_n(x_n)$$
and

$$A(x) = A_1(x_1)\cdot ... \cdot A_n(x_n)$$
one has

$$C_{n, d, A}f(x) = C_{1, d, A_1} f_1(x_1) \cdot ... \cdot C_{1, d, A_n} f_n(x_n).$$
To motivate the introduction of the operators $C_{n,F,A}$ one just has to recall the context in which the original Calder\'{o}n commutators appeared \cite{calderon}, \cite{calderon1}, \cite{meyerc}. If one tries to extend Calder\'{o}n's algebra to $\R^n$ and to 
include in it pseudodifferential operators containing partial derivatives, one is naturally led to the study of the operators in (\ref{1}) and their natural generalizations.

It is clear and very well known that to prove statements such as the one in Theorem \ref{main}, one needs to prove {\it polynomial bounds} for the corresponding Calder\'{o}n commutators $C_{n,d,A}$. More specifically, Theorem \ref{main} reduces to
the estimate

\begin{equation}\label{poli}
\|C_{n, d, A} f \|_p \leq C(n,d)\cdot C(p) \cdot \|f\|_p \cdot \left\|\frac{\partial^n A}{\partial x_1 ... \partial x_n}\right\|_{\infty}^d
\end{equation}
for any $f\in L^p$, where $C(n,d)$ grows at most polynomially in $d$ \footnote{This reduction is a simple consequence of the fact that if one writes the analytic function $F$ as a power series, the generic operator $C_{n,F,A}$ itself becomes
a series involving all the commutators $C_{n, d, A}$. The polynomial bounds are necessary for this series to be absolutely convergent. }.

The argument of \cite{j1} to prove the {\it small $L^{\infty}$ norm} theorem used an induction on the dimension $n$. We work instead directly in $\R^n$ and since our method is essentially similar in every dimension, to keep the technicalities to a minimum,
we chose for the reader's convenience to describe the proof of the main Theorem \ref{main} in the particular case of the plane $\R^2$. However, it will be clear that the same proof works equally well in every dimension.

So from now on $n=2$ and the goal is to prove the corresponding (\ref{poli}). The operators $C_{2, d, A}$ that we would like to understand, are given by

\begin{equation}\label{commutators}
C_{2, d, A} f(x) = p.v. \int_{\R^2} f(x+t) \left(\frac{\Delta_{t_1}^{(1)}}{t_1}\circ\frac{\Delta_{t_2}^{(2)}}{t_2} A(x) \right)^d \frac{d t_1}{t_1}\frac{d t_2}{t_2}.
\end{equation}
If $a:= \frac{\partial^2 A}{\partial x_1 \partial x_2}$, then one observes that

\begin{equation}\label{ave}
\frac{\Delta_{t_1}^{(1)}}{t_1}\circ\frac{\Delta_{t_2}^{(2)}}{t_2} A(x) = \int_{[0,1]^2} a (x_1 + \alpha t_1, x_2 + \beta t_2) d \alpha d \beta.
\end{equation}
As in \cite{camil3}, using (\ref{ave}) $d$ times, one can see that if $a$ and $f$ are Schwartz functions, the implicit limit in (\ref{commutators}) exists and can be rewritten as

\begin{equation}\label{multiplier}
\int_{\R^{2d+2}} m_{2,d}(\xi, \xi_1, ..., \xi_d, \eta,\eta_1, ..., \eta_d) 
\widehat{f}(\xi, \eta)
\widehat{a}(\xi_1, \eta_1) ...
\widehat{a}(\xi_d, \eta_d)\cdot
\end{equation}

$$
\cdot e^{2\pi i (x_1, x_2)\cdot [(\xi, \eta) + (\xi_1,  \eta_1) + ... + (\xi_d,  \eta_d)]}
d \xi d \xi_1 ... d\xi_d
d \eta d \eta_1 ... d \eta_d
$$
where

\begin{equation}\label{product}
m_{2,d}(\xi, \xi_1, ..., \xi_d, \eta,\eta_1, ..., \eta_d):= m_{1,d}(\xi, \xi_1, ..., \xi_d)\cdot m_{1,d}(\eta,\eta_1, ..., \eta_d)
\end{equation}
with $m_{1,d}(\xi, \xi_1, ..., \xi_d)$ and $m_{1,d}(\eta,\eta_1, ..., \eta_d)$ given by

$$\int_{[0,1]^d}\sgn (\xi + \alpha_1\xi_1 + ... +\alpha_d \xi_d) d \alpha_1  ... d\alpha_d$$
and 

$$
\int_{[0,1]^d}\sgn (\eta + \beta_1 \eta_1 + ... + \beta_d \eta_d) d \beta_1 ... d \beta_d$$
respectively. Because of the formula (\ref{multiplier}) $C_{2,d}$ can be seen as a $(d+1)$-linear operator. However, it is important to realize (as in \cite{camil3}) that even though its symbol $m_{2,d}$ has the nice {\it product structure} in (\ref{product}),
it is not a classical {\it bi-parameter} symbol, since $m_{1,d}$ itself is not a classical Marcinkiewicz H\"{o}rmander Mihlin multiplier \footnote{$m_{1,d}$ is of course the symbol of the one dimensional $d$th Calder\'{o}n commutator \cite{camil3}.}.
As a consequence of this fact, the general {\it polydisc Coifman Meyer theorem} proved in \cite{mptt:biparameter}, \cite{mptt:multiparameter} cannot be applied in this case. The strategy would be to combine the techniques of 
\cite{mptt:biparameter}, \cite{mptt:multiparameter} with the new ideas of \cite{camil2}, \cite{camil3} and to show that (together with some other logarithmical estimates that will be proved in this paper) they are enough to obtain the polynomial bounds of 
(\ref{poli}). Given these remarks, it would clearly be of great help for the reader, to be already familiar with our earlier arguments in \cite{camil2}, \cite{camil3}.

We will prove the following 

\begin{theorem}\label{main1}
Let $1< p_1, ..., p_{d+1} \leq \infty$ and $1\leq p < \infty$ be so that $1/p_1+ ... + 1/p_{d+1} = 1/p$. Denote by $l$ the number of indices $i$ for which $p_i\neq \infty$. The operator $C_{2,d}$ extends naturally as a $(d+1)$-linear operator bounded from
$L^{p_1}\times ... \times L^{p_{d+1}} \rightarrow L^p$ with an operatorial bound of type

\begin{equation}
C(d)\cdot C(l)\cdot C(p_1)\cdot ... \cdot C(p_{d+1})
\end{equation}
where $C(d)$ grows at most polynomially in $d$ and $C(p_i) = 1$ as long as $p_i=\infty$ for $1\leq i\leq d+1$.
\end{theorem}

The above Theorem \ref{main1} is the bi-parameter extension of the corresponding Theorem 1.1 in \cite{camil3}. If we assume it for a moment, we see that (\ref{poli}) follows from it by taking $p_1=p$ and $p_2= ... = p_{d+1} = \infty$.

To show Theorem \ref{main1} we will prove that for every $1\leq i \leq d+2$ and for every $\phi_1, ..., \phi_{d+1}$ Schwartz functions, one has

\begin{equation}\label{schwartz}
\|C_{2,d}^{\ast i} (\phi_1, ..., \phi_{d+1})\|_p \leq C(d)\cdot C(l)\cdot C(p_1)\cdot ... \cdot C(p_{d+1})\cdot \|\phi_1\|_{p_1} \cdot ... \cdot \|\phi_{d+1}\|_{p_{d+1}}
\end{equation}
where $(p_j)_{j=1}^{d+1}$ and $p$ are as before and $(C_{2,d}^{\ast i})_{i=1}^{d+2}$ are the adjoints of the multilinear operator $C_{2,d}$.\footnote{For symmetry, we aso use the notation $C_{2,d} = C_{2,d}^{\ast d+2}$.}
Standard density and duality arguments as in \cite{camil3}, allow then one to conclude that the estimates in (\ref{schwartz}) can be naturally extended to arbitrary products of $L^{p_j}$ and $L^{\infty}$ spaces.\footnote{The reader is also referred to
our earlier \cite{camil3} for an explanation of why does one need the larger set of estimates in (\ref{schwartz}) for $C_{2,d}$ and its adjoints, even though one is interested in the more particular (\ref{poli}).}

Our plan for the rest of the paper is as follows. In the next section, Section 2, we describe some discrete model operators whose analysis will play an important role in understanding (\ref{schwartz}). In Section 3 we prove that
the main estimates (\ref{schwartz}) can be reduced to a general theorem for the model operators. In Section 4 we prove the theorem for the discrete model operators of Section 2. In Section 5 we show
logarithmical bounds for some shifted Hardy-Littlewood-Paley hybrid operators, which appear naturally in the study of the previous discrete models. Finally, in Section 6 we describe various generalizations of the main
Theorem \ref{main}.

{\bf Acknowledgements:} The present work has been partially supported by the NSF.

\section{Discrete model operators}

As mentioned earlier, the main task here would be to describe some discrete model operators, whose analysis is deeply related to the analysis on (\ref{schwartz}). Because of the formula (\ref{multiplier}), we now know that

\begin{equation}\label{tensor}
C_{2,d} = C_{1,d}\otimes C_{1,d}
\end{equation}
and so one should not be at all surprised, to find out that these {\it bi-parameter} model operators that will be introduced, are in fact tensor products of the {\it one-parameter} discrete model operators of \cite{camil3}.
And also as in \cite{camil3}, these operators are not going to be $(d+1)$-linear, but $l$-linear instead, for some $1\leq l \leq d+1$. The explanation for this is similar to the one in \cite{camil3}. To be able to prove (\ref{schwartz}),
one first decomposes $C_{2,d}$ into polynomially (in $d$) many {\it bi-parameter paraproduct like pieces} and then estimate each such piece independently on $d$. To be able to achieve this, one has first to realize that one can estimate
most of the $L^{\infty}$ functions easily by their $L^{\infty}$ norms and reduce (\ref{schwartz}) in this way to the corresponding estimate for some {\it minimal} $l$-linear operators. To prove the desired bounds for these minimal operators,
one has to interpolate between some Banach and quasi-Banach estimates, as in \cite{camil3}. The Banach estimates are easy, but the quasi-Banach estimates are hard. One has to discretize the operators carefully, in order to understand them completely.
And this is (in a few words) how one arrives at the model operators. Their definition is as follows.

A smooth function $\Phi (x)$ of one variable is said to be a bump function adapted to a dyadic interval $I$, if and only if one has

$$\left|\partial^{\alpha}\Phi (x)\right| \lesssim \frac{1}{|I|^{\alpha}} \frac{1}{\left(1+\frac{\text{dist} (x, I)}{|I|}\right)^{M}}$$
for all derivatives $\alpha$ satisfying $|\alpha| \leq 5$ and any large $M>0$ with the implict constants depending on it. 
Then, if $1\leq q \leq \infty$, we say that $|I|^{-1/q} \Phi$ is an $L^q$ normalized bump adapted to $I$. The function $\Phi(x)$ is said to be of {\it $\Psi$ type} if $\int_{\R} \Phi(x) d x  = 0$,
otherwise is said to be of {\it $\Phi$ type}.

A smooth function $\Phi (x, y)$ of two variables is said to be a bump function adapted to the dyadic rectangle $R = I \times J$ if and only if it is of the form $\Phi (x, y) = \Phi_1(x) \cdot \Phi_2 (y)$ with $\Phi_1(x)$ adapted to $I$ and 
$\Phi_2(y)$ adapted to $J$. If $I$ is a dyadic interval and $\n$ an integer, we denote by $I_\n := I + n |I|$ the dyadic interval having the same length as $I$ but sitting $\n$ units of length $|I|$ away from it.

Fix now $1\leq l \leq d+1$ and $\n_1 = (\n_1^1, \n_1^2), ..., \n_l = (\n_l^1, \n_l^2)$ arbitrary pairs of integers. Define also $\n_{l+1}:= (0,0)$. Consider families $(\Phi^j_{R_{\n_j}})_R$ for $1\leq j \leq l+1$ of $L^2$ normalized
bump functions adapted to dyadic rectangles $R_{\n_j} = I_{\n_j^1} \times J_{\n_j^2}$ where $R = I\times J$ runs inside a given finite collection $\r$ of dyadic rectangles in the plane.
Assume also that at least two of the families $(\Phi^j_{I_{\n_j^1}})_I$ for $1\leq j \leq l+1$ are of {\it $\Psi$ type} and that the same is true for the families $(\Phi^j_{J_{\n_j^2}})_J$ for $1\leq j \leq l+1$.

The discrete model operator associated to these families of functions is defined by

\begin{equation}\label{model}
T_\r (f_1, ..., f_l) = \sum_{R\in \r}
\frac{1}{|R|^{(l-1)/2}}
\langle f_1, \Phi^1_{R_{\n_1}}\rangle ...
\langle f_l, \Phi^l_{R_{\n_l}}\rangle
\Phi^{l+1}_R.
\end{equation}

The following theorem holds.

\begin{theorem}\label{discrete}
For any such a finite family of arbitrary dyadic rectangles, the $l$-linear operator $T_{\r}$ maps $L^{p_1}\times ... \times L^{p_l} \rightarrow L^p$ boundedly, for
any $1< p_1, ..., p_l < \infty$ with $1/p_1 + ... 1/p_l = 1/p$ and $0 < p < \infty$, with a bound of type

\begin{equation}
O\left(\prod_{j=1}^2 \log^2 <\n_1^j>\cdot ... \cdot \log^2 <\n_l^j>\right)
\end{equation}
where in general, $<m>$ simply denotes $2+ |m|$. And the implicit constants are allowed to depend on $l$.
\end{theorem}
This theorem is the bi-parameter generalization of Theorem 3.1 in \cite{camil3}. As pointed out there, standard arguments based on scale invariance and interpolation, allows one to reduce the above Theorem \ref{discrete}
to the more precise statement that for every $f_j\in L^{p_j}$ with $\|f_j\|_{p_j} = 1$ and measurable set $E\subseteq \R^2$ of measure $1$, there exists a subset $E'\subseteq E$ with $|E'| \sim |E|$ so that

\begin{equation}\label{precise}
\sum_{R\in \r}
\frac{1}{|R|^{(l-1)/2}}
|\langle f_1, \Phi^1_{R_{\n_1}}\rangle| ...
|\langle f_l, \Phi^l_{R_{\n_l}}\rangle|
|\langle f_{l+1}, \Phi^{l+1}_R\rangle| \lesssim 
\prod_{j=1}^2 \log^2 <\n_1^j>\cdot ... \cdot \log^2 <\n_l^j>
\end{equation}
where $f_{l+1}:= \chi_{E'}$. As in \cite{camil3}, the fact that one looses only logarithmical bounds in the above estimates, will be of a crucial importance later on.

\section{Reduction to the model operators}

The goal of this section is to show that indeed (\ref{schwartz}) can be reduced to Theorem \ref{discrete} or more precisely to its weaker but more precise variant (\ref{precise}).
In particular, one can find here a description of all the ideas that are necessary to understand why it is possible to estimate the biparameter Calder\'{o}n commutators $C_{2,d}$ with bounds that grow at most
polynomially in $d$. 

The reader familiar with our previous work will realize that this section is in fact a {\it tensor product} of the corresponding section in \cite{camil3} with itself. As there, the first task is to decompose $C_{2,d}$ into
polynomially many {\it biparameter paraproduct like pieces} which will be studied later on.

\subsection*{Non-compact and compact Littlewood-Paley decompositions}

Let $\Phi(x)$ be a Schwartz function which is even, positive and satisfying $\int_{\R}\Phi(x) d x  = 1$. Define also $\Psi(x)$ by

$$\Psi(x) = \Phi(x) - \frac{1}{2} \Phi(\frac{x}{2})$$
and observe that $\int_{\R}\Psi(x) d x = 0$.

Then, as always, consider the functions $\Psi_k(x)$ and $\Phi_k(x)$ defined by $2^k\Psi(2^k x)$ and $2^k\Phi(2^k x)$ respectively, for every integer $k\in\Z$. Notice also that all the $L^1$ norms of $\Phi_k$ are equal to $1$.
Since $\Psi_k(x) = \Phi_k(x) - \Phi_{k-1}(x)$
one can see that

$$\sum_{k\leq k_0} \Psi_k = \Phi_{k_0}$$
and so

$$\sum_{k\in\Z} \Psi_k = \delta_0$$
or equivalently

\begin{equation}\label{a}
\sum_{k\in\Z} \widehat{\Psi}_k (\xi) = 1
\end{equation}
for almost every $\xi\in \R$. On the other hans, as observed in \cite{camil3}, since $\widehat{\Psi}(0) = \widehat{\Psi}'(0) = 0$ one can write $\widehat{\Psi}(\xi)$ as

$$\widehat{\Psi}(\xi) = \xi^2 \varphi(\xi)$$
for some other smooth and rapidly decaying function $\varphi$.

These are what we called the non-compact (in frequency) Littlewood-Paley decompositions. The compact ones are obtained similarly, the only difference being that instead of considering the Schwartz function
$\Phi$ before, one starts with another one having the property that $\supp \widehat{\Phi} \subseteq [-1,1]$ and $\widehat{\Phi}(0) = 1$.

As explained in \cite{camil3}, the advantage of the non-compact Littlewood-Paley projections is reflected in the perfect estimate

\begin{equation}\label{perfect}
\left|f\ast \Phi_k(x) \right| \leq \|f\|_{\infty}
\end{equation}
which plays an important role in the argument.

\subsection*{The generic decomposition of $C_{2,d}$}

Using (\ref{multiplier}), if $f, f_1, ..., f_{d+1}$ are all Schwartz functions, one can write the $(d+2)$-linear form associated to $C_{2,d}$ as

\begin{equation}\label{b}
\int_{\xi+\xi_1+ ... + \xi_{d+1} = 0 \atop \eta + \eta_1 + ... +\eta_{d+1} = 0}
\left( \int_{[0,1]^d}\sgn (\xi + \alpha_1\xi_1 + ... +\alpha_d \xi_d) d \alpha_1  ... d\alpha_d   \right)\cdot
\end{equation}
$$
\left( \int_{[0,1]^d}\sgn (\eta + \beta_1 \eta_1 + ... + \beta_d \eta_d) d \beta_1 ... d \beta_d \right)\cdot
$$

$$
\widehat{f}(\xi, \eta)
\widehat{f_1}(\xi_1, \eta_1) ...
\widehat{f_{d+1}}(\xi_{d+1}, \eta_{d+1})d \xi d \xi_1 ... d\xi_{d+1}
d \eta d \eta_1 ... d \eta_{d+1}.
$$

Then, by using the Littlewood-Paley decompositions in (\ref{a}) several times one can write

\begin{equation}\label{c}
1 = \sum_{l_0, l_1, ..., l_{d+1} \in \Z}
\widehat{\Psi}_{l_0}(\xi)
\widehat{\Psi}_{l_1}(\xi_1) ...
\widehat{\Psi}_{l_d}(\xi_d)
\widehat{\Psi}_{l_{d+1}}(\xi_{d+1}).
\end{equation}
As in \cite{camil3} since for every $(d+2)$ tuple $(l_0, l_1, ..., l_{d+1})\in \Z^{d+2}$ one has that either $l_0 \geq l_1, ..., l_{d+1}$ or $l_1\geq l_0, ..., l_{d+1}$ ... or $l_{d+1}\geq l_0, ..., l_d$,
fixing always the biggest parameter and summing over the rest of them, one can rewrite (\ref{c}) as

\begin{equation}\label{d}
\begin{array}{l}
\sum_{l}
\widehat{\Psi}_{l}(\xi)
\widehat{\Phi}_{l}(\xi_1) ...
\widehat{\Phi}_{l}(\xi_d)
\widehat{\Phi}_{l}(\xi_{d+1}) + \\
... + \\
\sum_{l}
\widehat{\Phi}_{l}(\xi)
\widehat{\Phi}_{l}(\xi_1) ...
\widehat{\Phi}_{l}(\xi_d)
\widehat{\Psi}_{l}(\xi_{d+1}).
\end{array}
\end{equation}

Also as in \cite{camil3}, we use in (\ref{d}) compact Littlewood-Paley decompositions for the $\xi$ and $\xi_{d+1}$ variables and non-compact ones for the rest of them.
Every single term in the decomposition (\ref{d}) contains only one $\Psi$ type of a function and we would like to have (at least) two. To be able to producte another one, one has to recall that $\xi+\xi_1+...+\xi_{d+1} = 0$.
Taking this into account, let us take a look at the second (for instance) term in (\ref{d}) in the particular case when $l=0$. We rewrite is for simplicity as

\begin{equation}\label{e}
\widehat{\Phi}(\xi)
\widehat{\Psi}(\xi_1) ...
\widehat{\Phi}(\xi_d)
\widehat{\Phi}(\xi_{d+1}).
\end{equation}
We know from before that $\widehat{\Psi}(\xi_1) = \xi_1^2 \widehat{\varphi}(\xi_1)$ and so we can write 

$$\widehat{\Psi}(\xi_1) = \xi_1 \widehat{\varphi}(\xi_1)(-\xi - \xi_2 - ... - \xi_{d+1}) = -\xi_1\xi\widehat{\varphi}(\xi_1) - \xi_1 \xi_2 \widehat{\varphi}(\xi_1) - ... -\xi_1 \xi_{d+1} \widehat{\varphi}(\xi_1).
$$
Using this in (\ref{e}) allows one to decompose it as another sum of $O(d)$ terms, containing this time two function of $\Psi$ type, since besides $\xi_1 \widehat{\varphi}(\xi_1)$ one finds now either a factor
of type $\xi \widehat{\Phi}(\xi)$ or of type $\xi_j\widehat{\Phi_j}(\xi_j)$ for some $j=2, ..., d+1$.

If one performs a similar decomposition for every scale $l\in \Z$ and each of the terms in (\ref{d}) one obtains a splitting of the function
$1_{\{\xi+\xi_1+...+\xi_{d+1}=0\}}$ as a sum of $O(d^2)$ expressions whose generic inner terms contain two functions of $\Psi$ type as desired.

Since we are this time in the biparameter setting, one has to decompose $1_{\{\eta+\eta_1+...+\eta_{d+1}=0\}}$ in a completely similar manner. Combining these two decompositions, allows us to rewrite the $(d+2)$ linear form of $C_{2,d}$ as

\begin{equation}\label{f}
\sum_{k_1, k_2\in \Z}
\int_{\xi+\xi_1+ ... + \xi_{d+1} = 0 \atop \eta + \eta_1 + ... +\eta_{d+1} = 0}
\left( \int_{[0,1]^d}\sgn (\xi + \alpha_1\xi_1 + ... +\alpha_d \xi_d) d \alpha_1  ... d\alpha_d   \right)\cdot
\end{equation}
$$
\left( \int_{[0,1]^d}\sgn (\eta + \beta_1 \eta_1 + ... + \beta_d \eta_d) d \beta_1 ... d \beta_d \right)\cdot
$$

$$
\widehat{\Phi_{k_1}^{1,0}}(\xi)
\widehat{\Phi_{k_1}^{1,1}}(\xi_1) ...
\widehat{\Phi_{k_1}^{1,d}}(\xi_d)
\widehat{\Phi_{k_1}^{1,d+1}}(\xi_{d+1})\cdot
$$

$$
\widehat{\Phi_{k_2}^{2,0}}(\eta)
\widehat{\Phi_{k_2}^{2,1}}(\eta_1) ...
\widehat{\Phi_{k_2}^{2,d}}(\eta_d)
\widehat{\Phi_{k_2}^{2,d+1}}(\eta_{d+1})\cdot
$$

$$
\widehat{f}(\xi, \eta)
\widehat{f_1}(\xi_1, \eta_1) ...
\widehat{f_{d+1}}(\xi_{d+1}, \eta_{d+1})d \xi d \xi_1 ... d\xi_{d+1}
d \eta d \eta_1 ... d \eta_{d+1},
$$
which completes our generic decomposition.

Recall that at least two of the families  $(\widehat{\Phi_{k_1}^{1,j}}(\xi_j))_{k_1}$ for $0\leq j\leq d+1$ are of $\Psi$ type and likewise at least two of the families 
$(\widehat{\Phi_{k_2}^{2,j}}(\eta_j))_{k_2}$ for $0\leq j\leq d+1$ are of $\Psi$ type as well. We denote those indices by $i_1, i_2$ and $j_1, j_2$ respectvely. There are several cases that one has to consider 
which correspond to the positions of these indices. We call an index {\it intermediate} if it is between $1$ and $d$ and {\it extremal} if it is either $0$ or $d+1$. In \cite{camil2} we essentially witnessed two cases.
Case $1$ was when at least one of the $\Psi$ positions corresponded to an intermediate index and Case $2$ was when both of the $\Psi$ positions were extremal. Since we now work in the biparameter setting, there are as a consequence
four possible cases of type Case $i$ $\otimes$ Case $j$ for $1\leq i, j\leq 2$.

\subsection*{Case $1$ $\otimes$ Case $1$}

Assume here that $i_1= j_1 = 0$ and $i_2 = j_2 = 1$. As mentioned earlier, the fact that $i_1=i_2 = 0$ is not important, they can be anywhere else in the interval $[0, d+1]$. Also, the fact that the intermediate indices $i_2$ and $j_2$ have been
chosen to be equal is not important either, but we chose them so for the simplicity of the notation. As in \cite{camil3} we would like now to expand the two implicit symbols in (\ref{f}).

As there, let us denote by $\widetilde{\xi}:= \xi + \alpha_2 \xi_2 + ... + \alpha_d \xi_d$ and by $\widetilde{\eta}:= \eta + \beta_2 \eta_2 + ... + \beta_d\eta_d$ and recall from \cite{camil3} that the idea is to treat the first symbol of
(\ref{f}) as being dependent on the variables $\xi_1$ and $\widetilde{\xi}$ and similarly the second symbol of (\ref{f}) as being dependent on $\eta_1$ and $\widetilde{\eta}$. Also, since most of our functions do not have compact support in
frequency, we need to consider some other compact Littlewood-Paley decompositions. We first write as in \cite{camil3}

$$1 = \sum_{l_0, l_1}
\widehat{\Psi}_l(\widetilde{\xi}) \widehat{\Psi}_{l_1}(\xi_1) = \sum_{l_0<<l_1} ... + \sum_{l_0\sim l_1} ... + \sum_{l_0 >> l_1} ...
$$
which can be rewritten as

\begin{equation}\label{g}
\sum_{r_1}\widehat{\Phi}_{r_1}(\widetilde{\xi}) \widehat{\Psi}_{r_1}(\xi_1) + 
\sum_{r_1}\widehat{\Psi}_{r_1}(\widetilde{\xi}) \widehat{\Psi}_{r_1}(\xi_1) + 
\sum_{r_1}\widehat{\Psi}_{r_1}(\widetilde{\xi}) \widehat{\Phi}_{r_1}(\xi_1).
\end{equation}
Then, we consider an identical decomposition, but for the variables $\widetilde{\eta}$ and $\eta_1$ this time, where the summation is indexed over the parameter $r_2$. If we insert (\ref{g}) into (\ref{f}) it becomes a sum
of three distinct expressions that generate the subcases $1_a$, $1_b$ and $1_c$ respectively. If in addition one inserts the analogous formula of (\ref{g}) for the variables $\widetilde{\eta}$ and $\eta_1$ into (\ref{f})
as well, one ends up with nine {\it biparameter subcases} of type Case $1_a \otimes$ Case $1_a$, Case $1_a\otimes$ Case $1_b$ and so on.

\subsection*{Case $1_a \otimes$ Case $1_a$} To analyze the impact that these extra decompositions have, we consider for simplicity (as in \cite{camil3}) the particular term corresponding to $k_1=k_2 = 0$. However, the argument we use is
scale invariant.

Let us ignore the symbol in (\ref{f}) for now and just concentrate on the remaining expression which becomes

\begin{equation}\label{h}
\left(
\sum_{r_1}\left[\widehat{\Phi}_{r_1}(\widetilde{\xi}) \widehat{\Psi}_{r_1}(\xi_1)\right]
\widehat{\Phi_{0}^{1,0}}(\xi)
\widehat{\Phi_{0}^{1,1}}(\xi_1) ...
\widehat{\Phi_{0}^{1,d}}(\xi_d)
\widehat{\Phi_{0}^{1,d+1}}(\xi_{d+1})
\right)\cdot
\end{equation}

$$
\left(
\sum_{r_2}\left[\widehat{\Phi}_{r_2}(\widetilde{\eta}) \widehat{\Psi}_{r_2}(\eta_1)\right]
\widehat{\Phi_{0}^{2,0}}(\eta)
\widehat{\Phi_{0}^{2,1}}(\eta_1) ...
\widehat{\Phi_{0}^{2,d}}(\eta_d)
\widehat{\Phi_{0}^{2,d+1}}(\eta_{d+1})
\right) =
$$

$$\left( \sum_{r_1 \leq 0} ... +   \sum_{r_1 > 0} ... \right)\cdot$$

$$\left( \sum_{r_2 \leq 0} ... +   \sum_{r_2 > 0} ... \right) :=$$

$$\left( 1'_a + 1''_a\right) \otimes \left( 1'_a + 1''_a\right)$$
which allows us to split our existing subcase into four additional subcases.

\subsection*{Case $1'_a$ $\otimes$ Case $1'_a$} So this corresponds to the situation when both $r_1$ and $r_2$ are negative. As in \cite{camil3}, using the fact that
$\widehat{\Psi}_{r_1}(\xi_1)$ is compactly supported and given that $\widehat{\Phi_{0}^{1,1}}(\xi_1)$ is also of $\Psi$ type (in fact it is of the form $\xi_1\widehat{\varphi}(\xi_1)$) one can rewrite {\it the $\xi$ part} of (\ref{h}) as

$$\sum_{r_1\leq 0}
2^{r_1} 
\widehat{\Phi}_{r_1}(\widetilde{\xi}) 
\widehat{\Phi_{0}^{1,0}}(\xi)
\widehat{\Psi_{r_1}^{1,1}}(\xi_1) ... 
\widehat{\Phi_{0}^{1,d}}(\xi_d)
\widehat{\Phi_{0}^{1,d+1}}(\xi_{d+1}) =
$$

$$
\sum_{r_1\leq 0}
2^{r_1} 
\left[\widehat{\widetilde{\Phi}}_{r_1}(\widetilde{\xi})\widehat{\widetilde{\Psi}_{r_1}^{1,1}}(\xi_1)\right]\cdot
\widehat{\Phi_{0}^{1,0}}(\xi)
\widehat{\Psi_{r_1}^{1,1}}(\xi_1) ... 
\widehat{\Phi_{0}^{1,d+1}}(\xi_{d+1})\widehat{\Phi}_{r_1}(\widetilde{\xi})
$$
for naturally chosen compactly supported functions $\widehat{\widetilde{\Phi}}_{r_1}(\widetilde{\xi})$, $\widehat{\widetilde{\Psi}_{r_1}^{1,1}}(\xi_1)$ and $\widehat{\Psi_{r_1}^{1,1}}(\xi_1)$.

This allows us to split the symbol

$$
\left(\int_0^1 \sgn (\widetilde{\xi} + \alpha_1 \xi_1) d \alpha_1 \right)
\widehat{\widetilde{\Phi}}_{r_1}(\widetilde{\xi})\widehat{\widetilde{\Psi}_{r_1}^{1,1}}(\xi_1)
$$
as a double Fourier series of the form

\begin{equation}\label{i}
\sum_{\widetilde{n}, \widetilde{n_1} \in \Z}
C^{r_1}_{\widetilde{n}, \widetilde{n_1}} e^{2\pi i \frac{\widetilde{n}}{2^{r_1}} \widetilde{\xi}} \cdot e^{2\pi i \frac{\widetilde{n_1}}{2^{r_1}} \xi_1} 
\end{equation}
where the Fourier coefficients satisfy the quadratic estimates

\begin{equation}\label{ii}
\left|C^{r_1}_{\widetilde{n}, \widetilde{n_1}}\right| = \left|C_{\widetilde{n}, \widetilde{n_1}}\right| \lesssim \frac{1}{<\widetilde{n}>^2} \frac{1}{<\widetilde{n_1}>^{\#}}
\end{equation}
for an arbitrarily large number $\# > 0$. See \cite{camil3} for these important estimates.

Clearly, there are similar calculations that one can make for {\it the $\eta$ part} of (\ref{h}). Using both of them, one can see that the particular contribution of $1'_a\otimes 1'_a$ in (\ref{f}) (at scale $1$) becomes

$$\int_{[0,1]^{d-1}} \int_{[0,1]^{d-1}}
\sum_{r_1 \leq 0} 2^{r_1}
\sum_{r_2 \leq 0} 2^{r_2}
\sum_{\widetilde{n}, \widetilde{n_1}} C^{r_1}_{\widetilde{n}, \widetilde{n_1}}
\sum_{\widetilde{\widetilde{n}}, \widetilde{\widetilde{n_1}}} C^{r_2}_{\widetilde{\widetilde{n}}, \widetilde{\widetilde{n_1}}}
\int_{\xi+ \xi_1 + ... + \xi_{d+1} =0 \atop \eta + \eta_1 + ... + \eta_{d+1} = 0}
$$

$$\left[\widehat{\Phi_{0}^{1,0}}(\xi) e^{2\pi i \frac{\widetilde{n}}{2^{r_1}} \xi} \cdot \widehat{\Phi_{0}^{2,0}}(\eta) e^{2\pi i \frac{\widetilde{\widetilde{n}}}{2^{r_2}} \eta} \right]\cdot $$

$$\left[\widehat{\Psi_{r_1}^{1,1}}(\xi_1) e^{2\pi i \frac{\widetilde{n_1}}{2^{r_1}} \xi_1} \cdot \widehat{\Psi_{r_2}^{2,1}}(\eta_1) e^{2\pi i \frac{\widetilde{\widetilde{n_1}}}{2^{r_2}} \eta_1} \right]\cdot $$

$$\left[\widehat{\Phi_{0}^{1,2}}(\xi_2) e^{2\pi i \frac{\widetilde{n}}{2^{r_1}} \alpha_2 \xi_2} \cdot \widehat{\Phi_{0}^{2,2}}(\eta_2) e^{2\pi i \frac{\widetilde{\widetilde{n}}}{2^{r_2}} \beta_2 \eta_2} \right]\cdot $$

$$ ... $$

$$\left[\widehat{\Phi_{0}^{1,d}}(\xi_d) e^{2\pi i \frac{\widetilde{n}}{2^{r_1}} \alpha_d \xi_d} \cdot \widehat{\Phi_{0}^{2,d}}(\eta_d) e^{2\pi i \frac{\widetilde{\widetilde{n}}}{2^{r_2}} \beta_d \eta_d} \right]\cdot $$

$$\left[\widehat{\Phi_{0}^{1,d+1}}(\xi_{d+1}) \cdot \widehat{\Phi_{0}^{2,d+1}}(\eta_{d+1})\right]\cdot$$

\begin{equation}\label{j}
\left[\widehat{\Phi}_{r_1}(\widetilde{\xi}) \cdot \widehat{\Phi}_{r_2}(\widetilde{\eta})\right]\cdot
\end{equation}

$$\widehat{f}(\xi, \eta)
\widehat{f_1}(\xi_1, \eta_1) ...
\widehat{f_{d+1}}(\xi_{d+1}, \eta_{d+1})d \xi d \xi_1 ... d\xi_{d+1}
d \eta d \eta_1 ... d \eta_{d+1} d \alpha_2 ... d\alpha_d d \beta_2 ... d\beta_d.
$$

Now, if one fixes $\vec{\alpha}$, $\vec{\beta}$, $r_1$, $r_2$, $\widetilde{n}$, $\widetilde{n_1}$, $\widetilde{\widetilde{n}}$, $\widetilde{\widetilde{n_1}}$, the corresponding inner expresion in (\ref{j}) becomes

\begin{equation}\label{k}
\int_{\xi+ \xi_1 + ... + \xi_{d+1} =0 \atop \eta + \eta_1 + ... + \eta_{d+1} = 0}
\end{equation}

$$\left[\widehat{f}(\xi, \eta)\cdot \widehat{\Phi_{0}^{1,0}}(\xi) e^{2\pi i \frac{\widetilde{n}}{2^{r_1}} \xi} \cdot \widehat{\Phi_{0}^{2,0}}(\eta) e^{2\pi i \frac{\widetilde{\widetilde{n}}}{2^{r_2}} \eta} \right]\cdot $$

$$\left[\widehat{f_1}(\xi_1, \eta_1)\cdot \widehat{\Psi_{r_1}^{1,1}}(\xi_1) e^{2\pi i \frac{\widetilde{n_1}}{2^{r_1}} \xi_1} \cdot \widehat{\Psi_{r_2}^{2,1}}(\eta_1) e^{2\pi i \frac{\widetilde{\widetilde{n_1}}}{2^{r_2}} \eta_1} \right]\cdot $$

$$ ... $$

$$\left[\widehat{f_{d}}(\xi_{d}, \eta_{d})\cdot\widehat{\Phi_{0}^{1,d}}(\xi_d) e^{2\pi i \frac{\widetilde{n}}{2^{r_1}} \alpha_d \xi_d} \cdot \widehat{\Phi_{0}^{2,d}}(\eta_d) e^{2\pi i \frac{\widetilde{\widetilde{n}}}{2^{r_2}} \beta_d \eta_d} \right]\cdot $$

$$\left[\widehat{f_{d+1}}(\xi_{d+1}, \eta_{d+1})\cdot \widehat{\Phi_{0}^{1,d+1}}(\xi_{d+1}) \cdot \widehat{\Phi_{0}^{2,d+1}}(\eta_{d+1})\right]\cdot$$

$$\widehat{\Phi}_{r_1}(\xi + \alpha_2 \xi_2 + ... + \alpha_d \xi_d)\cdot \widehat{\Phi}_{r_2}(\eta+ \beta_2\eta_2 + ... + \beta_d\eta_d)$$

$$d \xi d \xi_1 ... d\xi_{d+1} d \eta d \eta_1 ... d \eta_{d+1}.$$
To be able to continue the calculations we need the following lemma.

\begin{lemma}\label{average}
If $F, F_1, ..., F_{d+1}, \widetilde{\Phi}, \widetilde{\widetilde{\Phi}}$ are all Schwartz functions, then one has

$$
\int_{\xi+ \xi_1 + ... + \xi_{d+1} =0 \atop \eta + \eta_1 + ... + \eta_{d+1} = 0}
\widehat{F}(\xi, \eta)
\widehat{F_1}(\xi_1, \eta_1) ...
\widehat{F_{d+1}}(\xi_{d+1}, \eta_{d+1})\cdot
$$

$$\widehat{\widetilde{\Phi}}(a \xi + a_1 \xi_1 + ... + a_{d+1} \xi_{d+1})\cdot
\widehat{\widetilde{\widetilde{\Phi}}}(b\eta + b_1 \eta_1 + ... + \beta_{d+1}\eta_{d+1})
d \xi d \xi_1 ... d\xi_{d+1} d \eta d\eta_1 ... d\eta_{d+1} =
$$

$$\int_{\R^4} F(x_1 - a t_1, x_2 - a t_2) F_1(x_1-a_1t_1, x_2-b_1t_2) ... F_{d+1}(x_1-a_{d+1}t_1, x_2 - b_{d+1}t_2)
\widetilde{\Phi}(t_1)\widetilde{\widetilde{\Phi}}(t_2) d x_1 d x_2 d t_1 d t_2
$$ 
for $a, a_1, ..., a_{d+1}, b, b_1, ..., b_{d+1}$ arbitrary real numbers.
\end{lemma}

This Lemma \ref{average} is the biparameter extension of Lemma $4.1$ in \cite{camil3} and since its proof doesn't use any new ideas it is left to the reader. As pointed out in \cite{camil3}
there is also a natural generalization of it, which states that the formula works for more than two averages (so one can take an arbitrary number of $\widetilde{\Phi}$ functions and another arbitrary number of
$\widetilde{\widetilde{\Phi}}$ ones).

As in \cite{camil3}, if $G$ is now an arbitrary Schwartz function and $a$ a real number we denote by $G^a$ the function defined by

$$\widehat{G^a}(\xi) = \widehat{G}(\xi) e^{2\pi i a \xi}.$$
Alternatively, one has $G^a(x) = G(x-a)$. Using the above lemma and this notation the previous (\ref{k}) can be rewritten as

$$
\int_{\R^4}
\left( f\ast \Phi_0^{1,0,\frac{\widetilde{n}}{2^{r_1}}}\otimes \Phi_0^{2,0,\frac{\widetilde{\widetilde{n}}}{2^{r_2}}}\right)(x_1-t_1, x_2 - t_2)\cdot
$$

$$\left(f_1\ast \Psi_{r_1}^{1,1,\frac{\widetilde{n_1}}{2^{r_1}}}\otimes \Psi_{r_2}^{2,1,\frac{\widetilde{\widetilde{n_1}}}{2^{r_2}}}\right)(x_1, x_2)\cdot$$    

$$\prod_{j=2}^d \left(f_j \ast \Phi_0^{1,j,\frac{\widetilde{n}}{2^{r_1}}\alpha_j}\otimes \Phi_0^{2,j,\frac{\widetilde{\widetilde{n}}}{2^{r_2}}\beta_j}\right)(x_1-\alpha_j t_1, x_2 - \beta_j t_2)\cdot$$

$$\left(f_{d+1}\ast \Phi_0^{1,d+1}\otimes \Phi_0^{2,d+1}\right)(x_1, x_2)\cdot$$

$$\Phi_{r_1}(t_1) \Phi_{r_2}(t_2) d t_1 d t_2 d x_1 d x_2 = $$

$$
\int_{\R^4}
\left( f\ast \Phi_0^{1,0,\frac{\widetilde{n}}{2^{r_1}}}\otimes \Phi_0^{2,0,\frac{\widetilde{\widetilde{n}}}{2^{r_2}}}\right)(x_1-\frac{t_1}{2^{r_1}}, x_2 - \frac{t_2}{2^{r_2}})\cdot
$$

$$\left(f_1\ast \Psi_{r_1}^{1,1,\frac{\widetilde{n_1}}{2^{r_1}}}\otimes \Psi_{r_2}^{2,1,\frac{\widetilde{\widetilde{n_1}}}{2^{r_2}}}\right)(x_1, x_2)\cdot$$    

$$\prod_{j=2}^d \left(f_j \ast \Phi_0^{1,j,\frac{\widetilde{n}}{2^{r_1}}\alpha_j}\otimes \Phi_0^{2,j,\frac{\widetilde{\widetilde{n}}}{2^{r_2}}\beta_j}\right)(x_1-\frac{\alpha_j t_1}{2^{r_1}}, x_2 - \frac{\beta_j t_2}{2^{r_2}})\cdot$$

$$\left(f_{d+1}\ast \Phi_0^{1,d+1}\otimes \Phi_0^{2,d+1}\right)(x_1, x_2)\cdot$$

$$\Phi_{0}(t_1) \Phi_{0}(t_2) d t_1 d t_2 d x_1 d x_2 = $$

\begin{equation}\label{l}
\int_{\R^4}
\left( f\ast \Phi_0^{1,0,\frac{\widetilde{n}-t_1}{2^{r_1}}}\otimes \Phi_0^{2,0,\frac{\widetilde{\widetilde{n}}-t_2}{2^{r_2}}}\right)(x_1, x_2)\cdot
\end{equation}

$$\left(f_1\ast \Psi_{r_1}^{1,1,\frac{\widetilde{n_1}}{2^{r_1}}}\otimes \Psi_{r_2}^{2,1,\frac{\widetilde{\widetilde{n_1}}}{2^{r_2}}}\right)(x_1, x_2)\cdot$$    

$$\prod_{j=2}^d \left(f_j \ast \Phi_0^{1,j,\frac{\widetilde{n}-t_1}{2^{r_1}}\alpha_j}\otimes \Phi_0^{2,j,\frac{\widetilde{\widetilde{n}}-t_2}{2^{r_2}}\beta_j}\right)(x_1, x_2)\cdot$$

$$\left(f_{d+1}\ast \Phi_0^{1,d+1}\otimes \Phi_0^{2,d+1}\right)(x_1, x_2)\cdot$$

$$\Phi_{0}(t_1) \Phi_{0}(t_2) d t_1 d t_2 d x_1 d x_2. $$
We have to remember now that all the calculations so far have been made under the assumption that $k_1=k_2 = 0$, but they can clearly be performed in general and then the analogous formula of (\ref{l}) is

\begin{equation}\label{m}
\int_{\R^4}
\left( f\ast \Phi_{k_1}^{1,0,\frac{\widetilde{n}-t_1}{2^{k_1+r_1}}}\otimes \Phi_{k_2}^{2,0,\frac{\widetilde{\widetilde{n}}-t_2}{2^{k_2+r_2}}}\right)(x_1, x_2)\cdot
\end{equation}

$$\left(f_1\ast \Psi_{k_1+r_1}^{1,1,\frac{\widetilde{n_1}}{2^{k_1+r_1}}}\otimes \Psi_{k_2+r_2}^{2,1,\frac{\widetilde{\widetilde{n_1}}}{2^{k_2+r_2}}}\right)(x_1, x_2)\cdot$$    

$$\prod_{j=2}^d \left(f_j \ast \Phi_{k_1}^{1,j,\frac{\widetilde{n}-t_1}{2^{k_1+r_1}}\alpha_j}\otimes \Phi_{k_2}^{2,j,\frac{\widetilde{\widetilde{n}}-t_2}{2^{k_2+r_2}}\beta_j}\right)(x_1, x_2)\cdot$$

$$\left(f_{d+1}\ast \Phi_{k_1}^{1,d+1}\otimes \Phi_{k_2}^{2,d+1}\right)(x_1, x_2)\cdot$$

$$\Phi_{0}(t_1) \Phi_{0}(t_2) d t_1 d t_2 d x_1 d x_2. $$
In conclusion, if one denotes by $\vec{\alpha} = (\alpha_2, ..., \alpha_d)$ and by $\vec{\beta} = (\beta_2, ..., \beta_d)$ one sees that the part of (\ref{f}) that corresponds to Case $1'_a\otimes$ Case $1'_a$ can be written as

\begin{equation}\label{n}
\int_{[0,1]^{d-1}} \int_{[0,1]^{d-1}}
\left(
\sum_{r_1 \leq 0} 2^{r_1}
\sum_{r_2 \leq 0} 2^{r_2}
\sum_{\widetilde{n}, \widetilde{n_1}} C^{r_1}_{\widetilde{n}, \widetilde{n_1}}
\sum_{\widetilde{\widetilde{n}}, \widetilde{\widetilde{n_1}}} C^{r_2}_{\widetilde{\widetilde{n}}, \widetilde{\widetilde{n_1}}}\cdot 
C_{2,d}^{r_1, r_2,\widetilde{n}, \widetilde{n_1}, \widetilde{\widetilde{n}}, \widetilde{\widetilde{n_1}}, \vec{\alpha}, \vec{\beta}, t_1, t_2}\right)\cdot
\end{equation}

$$
\Phi_{0}(t_1) \Phi_{0}(t_2) d t_1 d t_2 d \vec{\alpha} d \vec{\beta}
$$
where $C_{2,d}^{r_1, r_2,\widetilde{n}, \widetilde{n_1}, \widetilde{\widetilde{n}}, \widetilde{\widetilde{n_1}}, \vec{\alpha}, \vec{\beta}, t_1, t_2}$ is the operator whose $(d+2)$ linear form is given by summing over $k_1, k_2$ the inner expressions
of (\ref{m}).

To prove (\ref{schwartz}) for (\ref{n}) one would need to prove it for the operators $C_{2,d}^{r_1, r_2,\widetilde{n}, \widetilde{n_1}, \widetilde{\widetilde{n}}, \widetilde{\widetilde{n_1}}, \vec{\alpha}, \vec{\beta}, t_1, t_2}$ with bounds that are 
summable over $r_1, r_2,\widetilde{n}, \widetilde{n_1}, \widetilde{\widetilde{n}}, \widetilde{\widetilde{n_1}}$ and integrable over $\vec{\alpha}, \vec{\beta}, t_1, t_2$. It is clear that
these operators are essentially biparameter paraproducts and therefore one expects that the method of \cite{mptt:biparameter}, \cite{mptt:multiparameter} should be used. That indeed will be the case, but on the other hand the appearence of all these
parameters mentioned earlier, have the role to {\it shift} the implicit bump functions which appear in their definitions and as a consequence this time one has to be very precise, when evaluates the size of their boundedness constants.

Since in our case away from the indices $0$ and $1$ all the bump functions are of $\Phi$ type, the idea is to use the perfect estimate (\ref{perfect}) (or rather, its biparameter variant) to bound all the functions
which are in $L^{\infty}$. To be more specific, let us denote as in \cite{camil3} by $S$ the set of all the indices $2\leq j\leq d$ for which $p_j\neq \infty$. Set $l:=|S|+2$ and freeze all the $L^{\infty}$ normalized Schwartz functions
$f_j$ corresponding to the indices in $\{2, ..., d\}\setminus S$. The new resulted operator is a {\it minimal} $l$-linear operator which will be denoted by
$C_{2,d}^{l, r_1, r_2,\widetilde{n}, \widetilde{n_1}, \widetilde{\widetilde{n}}, \widetilde{\widetilde{n_1}}, \vec{\alpha}, \vec{\beta}, t_1, t_2}$.

\subsection*{Shifted hybrid maximal and square functions}

It is now time to recall a few basic facts about biparameter paraproducts, to be able to go further. Consider two generic families of $L^1$ normalized bump functions $(\Phi_{k_1}^j)_{k_1}$ and $(\Phi_{k_2}^j)_{k_2}$ for $1\leq j\leq l+1$
so that in each of them for two indices the corresponding sequences are of $\Psi$ type.

An $l$-linear biparameter paraproduct is an $l$-linear operator whose $(l+1)$-linear form is given by

\begin{equation}\label{o}
\int_{\R^2}
\sum_{k_1, k_2\in \Z} \prod_{j=1}^{l+1}\left(f_j\ast \Phi_{k_1}^j\otimes \Phi_{k_2}^j\right) (x_1, x_2) d x_1 d x_2.
\end{equation}

Let us first assume that we are in a case similar to the one considered before and that the $\Psi$ functions appear for the indices $j=1,2.$ Then, one can estimate the absolute value of (\ref{o}) by

$$\int_{\R^2}
SS(f_1)(x_1, x_2)\cdot SS(f_2)(x_1, x_2) \cdot \prod_{j=2}^{l+1} MM(f_j)(x_1, x_2) d x_1 d x_2
$$
where in general $MM(f)(x_1, x_2)$ and $SS(f)(x_1, x_2)$ are defined by

$$MM(f)(x_1, x_2) = \sup_{k_1, k_2} \left| f\ast \Phi_{k_1}\otimes \Phi_{k_2}(x_1, x_2)\right|$$
and

\begin{equation}\label{p}
SS(f)(x_1, x_2) = \left( \sum_{k_1, k_2} \left| f\ast \Phi_{k_1}\otimes \Phi_{k_2}(x_1, x_2)\right|^2\right)^{1/2}
\end{equation}
respectively. In order for (\ref{p}) to make sense, we assume of course that both $(\Phi_{k_1})_{k_1}$ and $(\Phi_{k_2})_{k_2}$ there, are of $\Psi$ type. Since both $MM$ and $SS$ are known to be bounded in every $L^p$ space for
$1<p<\infty$, the above argument proves that our particular biparameter paraproduct in (\ref{o}) is bounded from $L^{p_1}\times ... \times L^{p_l}$ into $L^p$ as long as
$1/p_1 + ... + 1/p_l = 1/p$ and $1< p_1, ..., p_l, p <\infty$.

As one can imagine, the above {\it $l^2\times l^{2}\times l^{\infty}$} argument can be {\it twisted}, in which case one naturally obtains {\it hybrid maximal and square functions} of type $MS$ and $SM$ defined by

\begin{equation}\label{r}
MS(f)(x_1, x_2) = \sup_{k_1} \left(\sum_{k_2}  \left| f\ast \Phi_{k_1}\otimes \Phi_{k_2}(x_1, x_2)\right|^2\right)^{1/2}
\end{equation}
and

\begin{equation}\label{s}
SM(f)(x_1, x_2) = \left(\sum_{k_1} \sup_{k_2} \left| f\ast \Phi_{k_1}\otimes \Phi_{k_2}(x_1, x_2)\right|^2\right)^{1/2} 
\end{equation}
respectively. One has to assume in (\ref{r}) that the family $(\Phi_{k_2})_{k_2}$ is of $\Psi$ type and that $(\Phi_{k_1})_{k_1}$ is of $\Psi$ type in (\ref{s}), for both expressions to make sense.

As observed in \cite{mptt:multiparameter} all these hybrid operators are bounded in $L^p$ for $1<p<\infty$ as well and as a consequence, one can bound every biparameter paraproduct in arbitrary
products of $L^p$ spaces, as long as all of their indices are strictly between $1$ and $\infty$.

This discussion shows that in order to understand the operator $C_{2,d}^{l, r_1, r_2,\widetilde{n}, \widetilde{n_1}, \widetilde{\widetilde{n}}, \widetilde{\widetilde{n_1}}, \vec{\alpha}, \vec{\beta}, t_1, t_2}$
(and of course, all the other possible ones) one has to understand how to bound not only the above operators, but also their {\it shifted} analogs of type $M_{n_1}M_{n_1}$, $S_{n_1}S_{n_2}$,
$M_{n_1} S_{n_2}$ and $S_{n_1} M_{n_2}$ which are defined similarly, but with respect to the {\it shifted} functions $(\Phi^{\frac{n_1}{2^{k_1}}}_{k_1})_{k_1}$ and 
$(\Phi^{\frac{n_2}{2^{k_2}}}_{k_2})_{k_2}$.

In \cite{camil2} we understood completely the one-parameter shifted maximal and square functions $M_n$ and $S_n$ and proved their boundedness on $L^p$ spaces with operatorial bounds of type
$O(\log <n>)$ \footnote {The logarithmical bounds for $M_n$ are due to Stein and can be found in \cite{stein} Chapter II.}.

Now the arguments of \cite{camil2} and \cite{mptt:multiparameter} show that their hybrid biparameter analogs mentioned before, will also be bounded on $L^p$ spaces with operatorial bounds of type
$O(\log^2 <n_1> \log ^2<n_2>)$, as long as one can prove logarithmic bounds for the so called Fefferman-Stein inequality, namely

\begin{equation}\label{ss}
\left\|
\left(\sum_{j=1}^N |M_n f_j|^2 \right)^{1/2}\right\|_p \leq C_p \log ^2 <n>\cdot
\left\|\left(
\sum_{j=1}^N |f_j|^2 \right)^{1/2}\right\|_p
\end{equation}
which should hold true for every $1<p<\infty$.

This inequality will be proven in detail in a later section. Until then, we will use freely all these logarithmic bounds.

\subsection*{Banach estimates for $C_{2,d}^{l, r_1, r_2,\widetilde{n}, \widetilde{n_1}, \widetilde{\widetilde{n}}, \widetilde{\widetilde{n_1}}, \vec{\alpha}, \vec{\beta}, t_1, t_2}$}

Given the logarithmic bounds for the shifted maximal and square functions described earlier, it is not difficult to see (as in \cite{camil3}) that the operator 
$C_{2,d}^{l, r_1, r_2,\widetilde{n}, \widetilde{n_1}, \widetilde{\widetilde{n}}, \widetilde{\widetilde{n_1}}, \vec{\alpha}, \vec{\beta}, t_1, t_2}$ is indeed bounded from
$L^{s_1}\times ... \times L^{s_l}$ into $L^s$ as long as $1/s_1+ ... + 1/s_l = 1/s$ and $1<s_1, ..., s_l, s <\infty$ with operatorial bounds no greater than

\begin{equation}\label{t}
\left(C <r_1> <r_2> \log <\widetilde{n}> \log<\widetilde{\widetilde{n}}> \log <\widetilde{n_1}> \log<\widetilde{\widetilde{n_1}}> \log <[t_1]> \log < [t_2]> \right)^{2 l}.
\end{equation}

And this contribution is perfect, given the extra factors $2^{r_1}, 2^{r_2}$ that appeared before (recall that both $r_1, r_2$ are negative in our case) and the quadratic decay in
$\widetilde{n}, \widetilde{n_1}, \widetilde{\widetilde{n}}, \widetilde{\widetilde{n_1}}$.

\subsection*{Quasi-Banach estimates for $C_{2,d}^{l, r_1, r_2,\widetilde{n}, \widetilde{n_1}, \widetilde{\widetilde{n}}, \widetilde{\widetilde{n_1}}, \vec{\alpha}, \vec{\beta}, t_1, t_2}$}

Assume now that the index $s$ above satisfies $0<s<\infty$ and so it can be sub-unitary. We would like to estimate the boundedness constants of

\begin{equation}\label{tt}
C_{2,d}^{l, r_1, r_2,\widetilde{n}, \widetilde{n_1}, \widetilde{\widetilde{n}}, \widetilde{\widetilde{n_1}}, \vec{\alpha}, \vec{\beta}, t_1, t_2} : L^{s_1}\times ... \times L^{s_l} \rightarrow L^s.
\end{equation}

This time one has to discretize the operators in the $x_1, x_2$ variables and then take advantage of the general result in Theorem \ref{discrete}. Arguing as in \cite{camil3} we see that the problem reduces to estimating
expressions of type

$$\frac{1}{2^{6r_1 l}}\frac{1}{2^{6r_2 l}} \sum_{R} \frac{1}{|R|^{(l-1)/2}}
\left|\langle f, \Phi^{1,0}_{I_{[\widetilde{n}-t_1]}}\otimes \Phi^{2,0}_{J_{[\widetilde{\widetilde{n}}-t_2]}} \rangle\right|\cdot
\left|\langle f_1, \Phi^{1,1}_{I_{\widetilde{n_1}}}\otimes \Phi^{2,1}_{J_{\widetilde{\widetilde{n_1}}}} \rangle\right|\cdot
$$

$$\prod_{j\in S} 
\left|\langle f_j, \Phi^{1,j}_{I_{[(\widetilde{n}-t_1)\alpha_j]}}\otimes \Phi^{2,j}_{J_{[(\widetilde{\widetilde{n}}-t_2)\beta_j]}} \rangle\right|\cdot
\left|\langle f_{d+1}, \Phi^{1,d+1}_I \otimes \Phi^{2,d+1}_J \rangle \right|
$$
where the sum runs over dyadic rectangles of the form $R=I\times J$. By applying Theorem \ref{discrete} we see that the operatorial norms of (\ref{tt}) can be majorized by

$$\left(2^{-6r_1} 2^{-6r_2}
\log <\widetilde{n}> \log<\widetilde{\widetilde{n}}> \log <\widetilde{n_1}> \log<\widetilde{\widetilde{n_1}}> \log <[t_1]> \log < [t_2]> \right)^{2 l}
$$
and the same is true for all its adjoint operators. In the end, by using the same interpolation argument as in \cite{camil3}, one can see that
the operator $C_{2,d}^{l, r_1, r_2,\widetilde{n}, \widetilde{n_1}, \widetilde{\widetilde{n}}, \widetilde{\widetilde{n_1}}, \vec{\alpha}, \vec{\beta}, t_1, t_2}$ satisfies the inequality (\ref{schwartz}) with bounds that are clearly acceptable in
(\ref{n}) as desired.

These complete the discussion of Case $1'_a \otimes$ Case $1'_a$. The rest of the cases can be treated similarly after certain adjustments. Since all of these adjustments have been described carefully in \cite{camil3}, the only thing that is left is
to realize that they work equally well in our {\it tensor product} framework. The straightforward (but quite delicate) details are left to the reader.

\section{Proof of Theorem \ref{discrete}}

The proof of Theorem \ref{discrete} is based on the method developed in \cite{mptt:biparameter} and \cite{mptt:multiparameter}. First, we need to recall the following lemma whose detailed proof can be found in \cite{mptt:multiparameter}.

\begin{lemma}\label{desc}
Let $J\subseteq \R$ be an arbitrary interval. Then, every bump function $\phi_J$ adapted to $J$ can be written as

\begin{equation}
\phi_J = \sum_{k\in\N} 2^{-1000\, l \,k} \phi^k_J
\end{equation}
where for each $k\in\N$, $\phi^k_J$ is also a bump adapted to $J$ but with the additional property that $\supp (\phi^k_J)\subseteq 2^k J$
\footnote{$2^k J$ is the interval having the same center as $J$ and whose length is $2^k |J|$.}.
Moreover, if we assume $\int_{R}\phi_J(x) dx = 0$ then all the functions $\phi^k_J$ can be chosen so that $\int_{\R}\phi^k_J(x) dx = 0$
for every $k\in \N$.
\end{lemma}
Fix now the normalized functions $f_1, ..., f_l$ and the set $E$ as in (\ref{precise}). Using the above lemma, one can estimate the $(l+1)$- linear form on the left hand side of (\ref{precise}) as

\begin{equation}\label{split}
|\Lambda_{\r}(f_1, ..., f_{l+1})| \leq \sum_{\vec{k}\in\N^2} 2^{-1000 \, l \, |\vec{k}|}
\sum_{R\in \r}
\frac{1}{|R|^{(l-1)/2}}
|\langle f_1, \Phi^1_{R_{\n_1}}\rangle| ...
|\langle f_l, \Phi^l_{R_{\n_l}}\rangle|
|\langle f_{l+1}, \Phi^{l+1, \vec{k}}_R\rangle| 
\end{equation}
where the new functions $\Phi^{l+1,\vec{k}}_R$ have basically the same structure as the old $\Phi^{l+1}_R$
but they also have the additional property that $\supp ( \Phi^{l+1,\vec{k}}_R  )\subseteq 2^{\vec{k}}R$.
We denoted by $2^{\vec{k}}R:= 2^{k_1}I \times 2^{k_2}J$, $\vec{k}=(k_1, k_2)$ and 
$|\vec{k}|=k_1+k_2$.

As before, the form (\ref{split}) will be majorized later on by {\it tensorizing} two separate $l^2\times l^2\times l^{\infty}\times ... \times l^{\infty}$ estimates with respect to parameters $I$ and $J$.
As a consequence, for every index $1\leq j\leq l+1$ there are {\it hybrid square and maximal functions} naturally attached to that position which we denote by $(M-S)_j$. 
More specifically $(M-S)_j$ can be either the {\it discrete variant} of $M_{\n^1_j} M_{\n^2_j}$ or $S_{\n^1_j} S_{\n^2_j}$ or $M_{\n^1_j} S_{\n^2_j}$ or $S_{\n^1_j} M_{\n^2_j}$ depending on the positions of the corresponding $\Psi$ functions.
For simplicity, we do not write explicitly the dependence of these functions $(M-S)_j$ on the {\it shifting} parameters $\n_j$. Recall also that each of them comes with a boundedness constant which is no greater than
$O(\log^2 <\n_j^1> \log^2 <\n_j^2>)$\footnote{It is a {\it standard} fact that the continuous and the discrete variants of these operators behave similarly.}.

We construct now an exceptional set as follows. For each $\vec{k}\in\N^2$ define

\begin{equation}
\Omega_{-5|\vec{k}|}= \bigcup_{j=1}^l
\{ (x,y)\in\R^2 : (M-S)_j(f_j)(x,y)> C2^{5|\vec{k}|}  \log^2 <\n_j^1> \log^2 <\n_j^2>   \}.
\end{equation}
Also, define

\begin{equation}
\tilde{\Omega}_{-5|\vec{k}|}= \{ (x,y)\in\R^2 : MM(\chi_{\Omega_{-5|\vec{k}|}})(x,y) >\frac{1}{2 l} \}
\end{equation}
and then

\begin{equation}
\tilde{\tilde{\Omega}}_{-5|\vec{k}|}= \{ (x,y)\in\R^2 : MM(\chi_{\tilde{\Omega}_{-5|\vec{k}|}})(x,y) >\frac{1}{2^{|\vec{k}|}} \}.
\end{equation}
Finally, we denote by

$$\Omega = \bigcup_{\vec{k}\in\N^2}\tilde{\tilde{\Omega}}_{-5|\vec{k}|}.$$
It is clear that $|\Omega|< 1/2$ if $C$ is a big enough constant, which we fix from now on. Then, define $E':= E\setminus\Omega$ and observe that
$|E'|\sim 1$. 

Fix then $\vec{k}\in\N^2$ and look at the corresponding inner sum in (\ref{split}). We split it into two parts as follows.
Part I sums over those rectangles $R$ with the property that

\begin{equation}
R\cap\tilde{\Omega}_{-5|\vec{k}|}^c \neq \emptyset
\end{equation}
while Part II sums over those rectangles with the property that

\begin{equation}
R\cap\tilde{\Omega}_{-5|\vec{k}|}^c =\emptyset.
\end{equation}

We observe that Part II is identically equal to zero, because if $R\cap\tilde{\Omega}_{-5|\vec{k}|}^c \neq \emptyset$ then $R\subseteq\tilde{\Omega}_{-5|\vec{k}|}$
and in particular this implies that $2^{\vec{k}}R\subseteq \tilde{\tilde{\Omega}}_{-5|\vec{k}|}$ which is a set disjoint from $E'$. It is therefore enough to estimate Part I
only.

Since $R\cap\tilde{\Omega}_{-5|\vec{k}|}^c\neq\emptyset$, it follows that 
$\frac{|R\cap\Omega_{-5|\vec{k}|}|}{|R|}\leq \frac{1}{2 l}$ or equivalently,
$|R\cap\Omega^c_{-5|\vec{k}|}|> \frac{2 l -1}{2 l}|R|$. 

We are now going to describe $l+1$ decomposition procedures, one for each function $f_j$ for $1\leq j\leq l+1$. Later on, we will
combine them, in order to estimate our sum.

Independently, for every $1\leq j\leq l$, define

$$\Omega^j_{-5|\vec{k}|+1}= \{ (x,y)\in\R^2 : (M-S)_j(f_j)(x,y)> \frac{C 2^{5|\vec{k}|} \log^2 <\n_j^1> \log^2 <\n_j^2> }{2^1} \}$$
and set

$$\r^j_{-5|\vec{k}| +1}= \{ R\in \r : |R\cap\Omega^j_{-5|\vec{k}| +1}|>\frac{1}{2 l} |R| \},$$
then define

$$\Omega^j_{-5|\vec{k}| +2}= \{ (x,y)\in\R^2 : (M-S)_j(f_j)(x,y)> \frac{C 2^{5|\vec{k}|} \log^2 <\n_j^1> \log^2 <\n_j^2>  }{2^2} \}$$
and set

$$\r^j_{-5|\vec{k}| +2}= \{ R\in \r \setminus\r^j_{-5|\vec{k}| +1} : |R\cap\Omega^j_{-5|\vec{k}| +2}|>\frac{1}{2 l} |R| \},$$
and so on. The constant $C>0$ is the one in the definition of the set $E'$ above.
Since there are finitely many rectangles, this algorithm ends after a while, producing the sets $\{\Omega^j_{s_j}\}$
and $\{\r^j_{s_j}\}$ such that $\r =\cup_{s_j}\r_{s_j}$.

We would clearly like to have such a decomposition available for the last function $f_{l+1}$ as well. To do this, we first need to
construct the analogue of the set $\Omega_{-5|\vec{k}|}$, for it. Pick  $N>0$ a big enough integer such that for every
$R\in\r$ we have $|R\cap\Omega^{l+1 c}_{-N}|> \frac{2 l -1}{ 2 l} |R|$ where we defined

$$\Omega^{l+1}_{-N}= \{ (x,y)\in\R^2 : (M-S)_{l+1}(f_{l+1})(x,y)> C 2^N \}.$$
Then, similarly to the previous algorithms, we define

$$\Omega^{l+1}_{-N+1}= \{ (x,y)\in\R^2 : (M-S)_{l+1}(f_{l+1})(x,y)> \frac{C 2^N}{2^1} \}$$
and set

$$\r^{l+1}_{-N+1}= \{ R\in \r : |R\cap\Omega^{l+1}_{-N+1}|>\frac{1}{2 l} |R| \},$$
then define

$$\Omega^{l+1}_{-N+2}= \{ x\in\R^2 : (M-S)_{l+1}(f_{l+1})(x, y))> \frac{C 2^N}{2^2} \}$$
and set

$$\r^{l+1}_{-N+2}= \{ R\in \r\setminus\r^{l+1}_{-N+1} : |R\cap\Omega^{l+1}_{-N+2}|>\frac{1}{2 l} |R| \},$$
and so on, constructing the sets $\{\Omega^{l+1}_{s_{l+1}}\}$ and $\{\r^{l+1}_{s_{l+1}}\}$ such that $\r=\cup_{s_{l+1}}\r^{l+1}_{s_{l+1}}$.

Then we write Part I as

\begin{equation}\label{in5}
\sum_{s_1, ..., s_{l}>-5|\vec{k}|, s_{l+1}>-N} \sum_{R\in \r_{s_1, ..., s_{l+1}}}
\frac{1}{|R|^{(l+1)/2}}
|\langle f_1, \Phi^1_{R_{\n_1}}\rangle| ...
|\langle f_l, \Phi^l_{R_{\n_l}}\rangle|
|\langle f_{l+1}, \Phi^{l+1, \vec{k}}_R\rangle| |R|,
\end{equation}
where $\r_{s_1, ..., s_{l+1}}:= \r^1_{s_1}\cap ... \cap \r^{l+1}_{s_{l+1}}$. Now, if $R$ belongs to
 $\r_{s_1, ..., s_{l+1}}$   this means in particular that $R$ has not been selected at either of the previous $s_{j}-1$ steps for $1\leq j\leq l+1$, which means that $|R\cap\Omega^1_{s_1-1}|\leq \frac{1}{2 l} |R|$,
 ..., $|R\cap\Omega^{l+1}_{s_{l+1}-1}|\leq \frac{1}{2 l} |R|$ or equivalently
$|R\cap\Omega^{1 c}_{s_1-1}|>\frac{2 l -1}{2 l} |R|$, ..., $|R\cap\Omega^{l+1 c}_{s_{l+1}-1}|>\frac{2 l -1}{2 l} |R|$ . But this implies that

\begin{equation}\label{in6}
|R\cap\Omega^{1 c}_{s_1-1}\cap ... \cap \Omega^{l+1 c}_{s_{l+1}-1}     |> \frac{1}{2} |R|.
\end{equation}
In particular, using (\ref{in6}), the term in (\ref{in5}) is smaller than

$$\sum_{s_1, ..., s_l>-5|\vec{k}| , s_{l+1}>-N} \sum_{R\in \r_{s_1, ..., s_{l+1}}}
\frac{1}{|R|^{(l+1)/2}}
|\langle f_1, \Phi^1_{R_{\n_1}}\rangle| ...
|\langle f_l, \Phi^l_{R_{\n_l}}\rangle|
|\langle f_{l+1}, \Phi^{l+1, \vec{k}}_R\rangle|
|R\cap\Omega^{1 c}_{s_1-1}\cap ... \cap \Omega^{l+1 c}_{s_{l+1}-1}| =$$

$$\sum_{s_1, ...,  s_l> -5|\vec{k}|, s_{l+1}>-N} \int_{ \Omega^{1 c}_{s_1-1}\cap ... \cap \Omega^{l+1 c}_{s_{l+1}-1}   }
\sum_{R\in \r_{s_1, ..., s_{l+1}}}
\frac{1}{|R|^{(l+1)/2}}
|\langle f_1, \Phi^1_{R_{\n_1}}\rangle| ...
|\langle f_l, \Phi^l_{R_{\n_l}}\rangle|
|\langle f_{l+1}, \Phi^{l+1, \vec{k}}_R\rangle|
\chi_{R}(x,y)\, dx dy$$

$$\lesssim \sum_{s_1, ..., s_{l+1}> -5|\vec{k}|, s_{l+1}>-N} \int_{  \Omega^{1 c}_{s_1-1}\cap ... \cap \Omega^{l+1 c}_{s_{l+1}-1} \cap
\Omega_{\r_{s_1, ..., s_{l+1}}}}
\prod_{j+1}^{l+1} (M-S)_j(f_j)(x,y)\, dx dy$$

\begin{equation}\label{in7}
\lesssim \sum_{s_1, ..., s_{l+1}> -5|\vec{k}|, s_{l+1}>-N} 2^{5 l |\vec{k}|}
\prod_{j=1}^l \log^2 <\n_j^1> \log^2 <\n_j^2>
2^{-s_1}\cdot ... \cdot2^{-s_{l+1}} |\Omega_{\r_{s_1, ..., s_{l+1}}}|,
\end{equation}
where

$$\Omega_{\r_{s_1, ..., s_{l+1}}}:= \bigcup_{R\in\r_{s_1, ..., s_{l+1}}}R .$$
On the other hand we can write

$$|\Omega_{\r_{s_1, ..., s_{l+1}}}|\leq |\Omega_{\r^1_{s_1}}|\leq
|\{ (x,y)\in\R^2 : MM(\chi_{\Omega^1_{s_1}})(x,y)> \frac{1}{2 l} \}|$$

$$\lesssim |\Omega^1_{s_1}|= |\{ (x,y)\in\R^2 : (M-S)_1(f_1)(x,y)>\frac{C(\vec{k}, \n_1)}{2^{s_1}} \}|\lesssim 2^{s_1 p_1}.$$
Similarly, we have

$$|\Omega_{\r_{s_1, ..., s_{l+1}}}|\lesssim 2^{s_j p_j}$$
for every $1\leq j\leq l$ and also

$$|\Omega_{\r_{s_1, ..., s_{l+1}}}|\lesssim 2^{s_{l+1} \alpha},$$
for every $\alpha >1$. Here we used the fact that all the operators $(M-S)_j$ are bounded
on $L^s$ as long as $1<s< \infty$ and also that $|E'|\sim 1$.
In particular, it follows that

\begin{equation}\label{*}
|\Omega_{\r_{s_1, ..., s_{l+1}}}|\lesssim 2^{s_1 p_1 \theta_1}\cdot ... \cdot 2^{s_l p_l \theta_l} 2^{s_{l+1} \alpha \theta_{l+1}}
\end{equation}
for any $0\leq \theta_1, ..., \theta_{l+1} < 1$, such that $\theta_1+ ... + \theta_{l+1}= 1$.

Now we split the sum in (\ref{in7}) into

\begin{equation}\label{lastt}
\sum_{s_1, ..., s_l> -5|\vec{k}|, s_{l+1}>0} 2^{5 l |\vec{k}|}  \prod_{j=1}^l \log^2 <\n_j^1> \log^2 <\n_j^2>  2^{-s_1}\cdot  ... \cdot 2^{-s_{l+1}} |\Omega_{\r_{s_1, ..., s_{l+1}}}| 
\end{equation}

$$
+ \sum_{s_1, ..., s_l> -5|\vec{k}|, 0>s_{l+1}>-N} 2^{5 l |\vec{k}|}  \prod_{j=1}^l \log^2 <\n_j^1> \log^2 <\n_j^2>  2^{-s_1}\cdot ... \cdot 2^{-s_{l+1}} |\Omega_{\r_{s_1, ..., s_{l+1}}}|.
$$
To estimate the terms in (\ref{lastt}) we use the inequality (\ref{*}) as follows. First, we choose $\theta_1, ..., \theta_l$ small enough so that $1-p_j\theta_j > 0$ for every $1\leq j\leq l$.
Because of this, $\theta_{l+1}$ can become quite close to $1$. To estimate the first term in (\ref{lastt}) we pick $\alpha$ very close to $1$ so that $1-\alpha\theta_{l+1} > 0$, while
to estimate the second term we pick $\alpha$ large enough so that $1-\alpha\theta_{l+1} < 0$

With these choices, the sum in (\ref{lastt})
is at most $O(2^{100 \, l \, |\vec{k}|}  \prod_{j=1}^l \log^2 <\n_j^1> \log^2 <\n_j^2> )$ and this makes the expression in (\ref{split}) to be $O(\prod_{j=1}^l \log^2 <\n_j^1> \log^2 <\n_j^2> )$ as desired, after summing over $\vec{k}\in\N^2$.

This ends our proof.

\section{Logarithmical bounds for the shifted hybrid maximal and square functions}

To complete the proof of the main theorem, we need to demonstrate the logaritmical bounds that have been used for the {\it shifted hybrid maximal and square functions}. As we mentioned before, the arguments of \cite{camil2} and \cite{mptt:biparameter} show that
they would follow from the following logarithmical bound for the vector valued Fefferman-Stein inequality.

\begin{theorem}\label{log}
Let $n\in\Z$ be a fixed integer and denote by $M_n$ the shifted maximal operator associated to $n$. Then, one has

\begin{equation}\label{f-stein}
\|
(\sum_{j=1}^N |M_n f_j|^2 )^{1/2} \|_p \leq C_p \log^2 <n>\cdot
\| (
\sum_{j=1}^N |f_j|^2  )^{1/2} \|_p
\end{equation}
for every $N$ and any $1<p<\infty$.
\end{theorem}

\begin{proof}
The proof is a combination of the classical argument of Fefferman and Stein \cite{stein} with the new ideas from \cite{camil2}. A nice description of the Fefferman-Stein inequality
is in Workman \cite{john} and we follow that presentation closely. There are three cases. Clearly, $n$ is supposed to be large, otherwise there is nothing to prove. Assume also that $n$ is positive, since the negative case is completely similar.

\subsection*{Case $1$: $p=2$}

This case is very simple and it follows immediately from the theorem in \cite{camil2} which says that $M_n$ is bounded on $L^2$ (and in fact on any $L^p$) with an operatorial bound of type
$O(\log <n>)$.

\subsection*{Case $2$: $p > 2$}

To understand this case one first needs to observe the following lemma.

\begin{lemma}\label{weakl1}
The following inequality holds

\begin{equation}\label{weakl1l1}
\alpha \cdot \int_{\{x : M_n f(x) > \alpha \}} |\Phi(x)| d x \lesssim \int_{\R} |f(x)| \M_n \Phi(x) d x 
\end{equation}
for every $\alpha > 0$ and measurable functions $f$ and $\Phi$, where $\M_n$ has been defined to be

$$\M_n \Phi (x) := \sum_{k=0}^{[\log_2 n]} M_{- 2^k} \Phi (x).
$$
\end{lemma}

\begin{proof}
To prove this lemma we need to recall a few facts from \cite{camil2}. Denote by $I_n$ maximal dyadic intervals selected with the property that

\begin{equation}
\frac{1}{|I_n|} \int_{I_n} |f(x)| d x > \alpha.
\end{equation}
Clearly, they are all disjoint and their union is equal to $\{x : M f(x) > \alpha\}$. Each $I_n$ comes with $[\log_2 n]$ dyadic intervals of the same length {\it attached} to it, denoted by $I_n^1$, ..., $I_n^{[\log_2 n]}$.
More precisely, $I_n^k$ lies $2^k$ steps of length $|I_n|$ to the left of $I_n$. It has been observed in \cite{camil2} that

$$\{ x : M_n f(x) > \alpha\} \subseteq \bigcup_{I_n} I_n \cup I_n^1 \cup ... \cup I_n^{[\log_2 n]}.$$
Using these, one can majorize the left hand side of (\ref{weakl1l1}) by

$$\alpha\cdot \sum_{I_n} \sum_{k=1}^{[\log_2 n]} \int_{I_n^k} |\Phi(x)| d x$$

$$
\lesssim\sum_{I_n} \sum_{k=1}^{[\log_2 n]} ( \frac{1}{|I_n|} \int_{I_n} |f(y)| d y ) \cdot ( \int_{I_n^k} |\Phi(x)| d x )
$$

\begin{equation}\label{123}
 = \sum_{I_n} \sum_{k=1}^{[\log_2 n]} ( \int_{I_n} |f(y)| d y ) \cdot ( \frac{1}{|I_n|} \int_{I_n^k} |\Phi(x)| d x ).
\end{equation}
Now, for every $y\in I_n$ one can see that

$$ \frac{1}{|I_n|} \int_{I_n^k} |\Phi(x)| d x \lesssim M_{-2^k} \Phi (y).$$
Using this in (\ref{123}) one immediately obtains the desired (\ref{weakl1l1}).

\end{proof}

The result of the above lemma implies that

$$M_n : L^1(\R, \M_n \Phi d x) \rightarrow L^{1,\infty}(\R, |\Phi| d x).$$
Since $\Phi$ can be assumed not to be identically equal to zero, we know that $\M_n \Phi > 0$. In particular, we also have the trivial bound

$$M_n : L^{\infty}(\R, \M_n \Phi d x) \rightarrow L^{\infty}(\R, |\Phi| d x)$$
and by interpolation, we obtain the following $L^2$ estimate

\begin{equation}\label{l2obs}
\int_{\R} |M_n f (x)|^2 |\Phi(x)| d x \lesssim \int_{\R} |f(x)|^2 \M_n \Phi(x) d x.
\end{equation}

Coming back to the proof of Case $2$, since $p > 2$ we know that $q:= p/2 > 1$. By picking an appropriate $\|\Phi\|_{q'} = 1$ and relying on the previous (\ref{l2obs}), one can write

$$
\|(\sum_{j=1}^N |M_n f_j|^2 )^{1/2} \|_p^2 = \|\sum_{j=1}^N |M_n f_j|^2 \|_q = \int_{\R} (\sum_{j=1}^N |M_n f_j|^2) |\Phi| d x 
$$

$$\lesssim 
\int_{\R}
(\sum_{j=1}^N |f_j|^2) \M_n \Phi d x \lesssim \|\sum_{j=1}^N |f_j|^2 \|_q \cdot \|\M_n \Phi\|_{q'} \lesssim \|(\sum_{j=1}^N |f_j|^2)^{1/2} \|_p^2 \cdot \|\M_n\|_{q'\rightarrow q'}.
$$
On the other hand, from the definition of $\M_n$ and the result of \cite{camil2}, we know that

$$\|\M_n\|_{q'\rightarrow q'} \leq \sum_{k=1}^{[\log_2 n]} \|\M_{-2^k}\|_{q'\rightarrow q'} \lesssim \sum_{k=1}^{[\log_2 n]} \log 2^k \lesssim \log^2 <n>
$$
which completes Case $2$.

\subsection*{Case $3$: $1 < p < 2$}

The idea here is to prove the following end point case

\begin{equation}\label{endpoint}
\|
(\sum_{j=1}^N |M_n f_j|^2 )^{1/2} \|_{1,\infty} \lesssim \log^2 <n>\cdot
\|(\sum_{j=1}^N |f_j|^2  )^{1/2} \|_1
\end{equation}
directly and then to apply standard vector valued interpolation with the corresponding $L^2$ estimate.

To prove (\ref{endpoint}), let $\alpha > 0$ and denote by $F(x) := (\sum_{j=1}^N |f_j(x)|^2  )^{1/2}$. Select maximal dyadic intervals $I_n$ with the property

$$\frac{1}{|I_n|} \int_{I_n} F(x) d x > \alpha.
$$
As before, we like to think of each $I_n$ as being related to the dyadic interval $I$, having the same length as $I_n$ and lying $n$ steps of length $|I_n|$ to the left of it. If we denote by
$\Omega := \bigcup_I I_n$ one has as usual

\begin{equation}\label{set}
|\Omega|  = \sum_I |I_n| \leq \frac{1}{\alpha} \sum_I \int_{I_n} F(x) d x \leq \frac{1}{\alpha} \|F\|_1.
\end{equation}
Observe that $F \leq \alpha$ on $\Omega^c$ and also that

$$\alpha < \frac{1}{|I_n|} \int_{I_n} F(x) d x \leq 2 \alpha$$
because of the maximality of $I_n$.

Split now each $f_k$ as $f_k = f'_k + f''_k$ where $f'_k := f_k\chi_{\Omega^c}$ and $f''_k:= f_k\chi_{\Omega}$.

\subsection*{Contribution of $\{f'_k\}$}

One can write 

$$|\{ x : (\sum_j |M_n f'_j(x)|^2)^{1/2} > \alpha/2 \}| \leq \frac{1}{\alpha^2}\| (\sum_j |M_n f'_j(x)|^2 )^{1/2}\|_2^2$$

$$\lesssim \frac{1}{\alpha^2} \log^2 <n> \| (\sum_j |f'_j(x)|^2 )^{1/2}\|_2^2 \lesssim \log^2 <n> \frac{1}{\alpha^2} \int_{\Omega^c} F^2(x) d x  \leq \log^2 <n> \frac{1}{\alpha}\|F\|_1$$
as desired.

\subsection*{Contribution of $\{f''_k\}$}

To estimate the corresponding contribution for $\{f''_k\}$, we have to be a bit more careful. Define first the functions $g_k$ by

\begin{equation}
g_k:= \sum_I (\frac{1}{|I_n|} \int_{I_n} |f_k(x)| d x ) \cdot \chi_{I_n}
\end{equation}
and after that $G(x):= (\sum_j |g_j(x)|^2 )^{1/2}$. Fix $x\in I_n$ and observe that by Minkowski's inequality one can write

$$G(x) = \left (\sum_ j (\frac{1}{|I_n|} \int_{I_n} |f_j(y)| d y )^2 \right)^{1/2}\leq
\frac{1}{|I_n|} \int_{I_n}(\sum_j |f_j(y)|^2)^{1/2} d y = \frac{1}{|I_n|} \int_{I_n} F(y) d y \leq 2 \alpha .$$
Using that $G$ is supported in $\Omega$ and arguing as before, we have

$$|\{ x : (\sum_j |M_n g_j (x)|^2)^{1/2} > \alpha/2 \}| \lesssim \frac{1}{\alpha^2} \log^2 <n> \| (\sum_j |g_j(x)|^2 )^{1/2}\|_2^2 $$

$$
= \frac{1}{\alpha^2} \log^2 <n> \| G \|_2^2 \lesssim \log^2 <n> |\Omega| \leq \frac{1}{\alpha} \log^2 <n> \|F\|_1 .
$$
We would like now to compare $M_n f''_k (x)$ with $M_n g_k (x)$ if possible. Denote by $\widetilde{\Omega}$ the set

$$\widetilde{\Omega} := \bigcup_I 3 I_n \cup 3 I_n^1 \cup ... \cup 3 I_n^{[\log_2 n]}$$
and observe that $|\widetilde{\Omega}| \lesssim \log <n> |\Omega|$. We will prove that for every $x\in \widetilde{\Omega}^c$ one has

\begin{equation}\label{point}
M_n f''_k (x) \leq M_n g_k (x)
\end{equation}
and this will clearly allow us to reduce the contribution of $\{f''_k\}$ to the contribution of $\{g_k\}$ which has been understood earlier.
Fix $x\in \widetilde{\Omega}^c$ and $x\in J$ a dyadic interval, so that the corresponding $\frac{1}{|J_n|} \int_{J_n} |f''_k (y) d y$ is different from zero. In particular, $J_n$ has to intersect $\Omega$ which is the
support of $f''_k$. Suppose now that $I$ is so that $J_n \cap I_n \neq \emptyset$. Then, one must have $I_n\subseteq J_n$ as the other alternative $J_n\subseteq I_n$ is not possible since $x\in \widetilde{\Omega}^c$.
But his implies that 

$$\frac{1}{|J_n|} \int_{J_n} |f''_k(x)| d x = \frac{1}{|J_n|} \int_{J_n} |g_k(x)| d x$$
which is enough to guarantee (\ref{point}) and end our proof.

\end{proof}

\section{Generalizations}

The goal of this section is to point out that virtually all the earlier generalizations that we described in the previous \cite{camil2} and \cite{camil3}, 
have {\it natural} extensions in this {\it multi-parameter world} and can be proved by the same method. We give here just two {\it samples} and leave the rest (and the straightforward details) to the imaginative reader.
Suppose for simplicity that we are in $\R^2$ and denote by $D_1:= \frac{\partial}{\partial x_1}$ and $D_2:= \frac{\partial}{\partial x_2}$. A direct computation shows that the {\it double commutator} $[|D_2|, [|D_1|, A]]$
can be rewritten as

\begin{equation}
 [|D_2|, [|D_1|, A]] f(x) = p.v. \int_{\R^2} f(x+t) \left( \frac{\Delta^{(1)}_{t_1}}{t_1} \circ \frac{\Delta^{(2)}_{t_2}}{t_2} A(x)\right) \frac{ d t_1}{t_1} \frac{d t_2}{t_2}
\end{equation}
which is precisely the {\it bidisc} extension of the first commutator of Calder\'{o}n. There is of course a similar formula available in every dimension.

\begin{theorem}
Let $a_1, ..., a_n$ be real numbers, all different from zero. The expression 

$$p.v. \int_{\R^n} f(x+t )\left( \frac{\Delta_{a_1 t_1}^{(1)}}{t_1} \circ ... \circ\frac{\Delta_{a_n t_n}^{(n)}}{t_n} A(x)\right)\frac{d t_1}{t_1} ... \frac{d t_n}{t_n}$$
viewed as a bilinear map in $f$ and $\frac{\partial^n A}{\partial x_1 ... \partial x_n}$ is bounded from $L^p\times L^q$ into $L^r$ for every
$1<p, q \leq \infty$ with $1/p+1/q=1/r$ and $1/2 < r < \infty$.
\end{theorem}
The particular case $q= \infty$ is in Journ\'{e} \cite{j1} but the rest of the estimates seem to be new. 

Then, one can also observe by a direct calculation that

$$[ |D_2|,    [ |D_1|,  [|D_2|, [|D_1|, A]]]] f(x)$$ 

$$
=p.v. \int_{\R^4} f(x+t+s)
\left(\frac{\Delta_{t_1}^{(1)}}{t_1}\circ\frac{\Delta_{t_2}^{(2)}}{t_2}\circ\frac{\Delta_{s_1}^{(1)}}{s_1}\circ\frac{\Delta_{s_2}^{(2)}}{s_2} A(x)\right)
\frac{d t_1}{t_1}\frac{d t_2}{t_2}\frac{d s_1}{s_1}\frac{d s_2}{s_2}
$$
which is the {\it bidisc} analogue of an operator introduced in \cite{camil3}. As we promised, we record now the following theorem.

\begin{theorem}
Let $F$ be an analytic function on a disc of a certain radius centered at the origin in the complex plane and $A$ a complex valued function in $\R^2$ so that $\frac{\partial^4 A}{\partial x_1^2 \partial x_2^2} \in L^{\infty}(\R^2)$
with an $L^{\infty}$ norm strictly smaller than the radius of convergence of $F$. Then, the linear operator

$$f\rightarrow p.v. \int_{\R^4} f(x+t+s)
F \left(\frac{\Delta_{t_1}^{(1)}}{t_1}\circ\frac{\Delta_{t_2}^{(2)}}{t_2}\circ\frac{\Delta_{s_1}^{(1)}}{s_1}\circ\frac{\Delta_{s_2}^{(2)}}{s_2} A(x)\right)
\frac{d t_1}{t_1}\frac{d t_2}{t_2}\frac{d s_1}{s_1}\frac{d s_2}{s_2}
$$
is bounded on $L^p(\R^2)$ for every $1<p<\infty$.
\end{theorem}

\end{document}